\documentclass{amsart}
\usepackage[dvipsnames]{xcolor}
\usepackage{fullpage}
\usepackage{float}
\usepackage{colonequals}
\usepackage{microtype}
\usepackage[inkscapeexe=/Applications/Inkscape.app/Contents/MacOS/inkscape,inkscapearea=page]{svg}
\usepackage{macros}
\begin{document}
\author{Mohabat Tarkeshian}

\title{On the geometry of exponential random graphs and applications}

\date{\today}

%\thanks{The author wishes to thank }

\address{Mathematics and Statistics Department\\St. Francis Xavier University \\ 2323 Notre Dame Avenue \\ Antigonish, NS B2G 2W5}
\email{mtarkesh@stfx.ca}

\begin{abstract}  
In a seminal paper in 2009, Borcea, Br{\"a}nd{\'e}n, and Liggett described the connection between probability distributions and the geometry of their generating polynomials. 
Namely, they characterized that stable generating polynomials correspond to distributions with the strongest form of negative dependence. This motivates us to investigate other distributions that can have this property, and our focus is on random graph models. 
In this article, we will lay the groundwork to investigate Markov random graphs, and more generally exponential random graph models (ERGMs), from this geometric perspective. 
In particular, by determining when their corresponding generating polynomials are either stable and/or Lorentzian. 
The \textit{Lorentzian} property was first described in 2020 by Br{\"a}nd{\'e}n and Huh and independently by Anari, Oveis-Gharan, and Vinzant where the latter group called it the \textit{completely log-concave} property. 
The theory of \textit{stable} polynomials predates this, and is commonly thought of as the multivariate notion of real-rootedness. 
Br{\"a}nd{\'e}n and Huh proved that stable polynomials are always Lorentzian. 
Although it is a strong condition, verifying stability is not always feasible. We will characterize when certain classes of Markov random graphs are stable and when they are only Lorentzian.
We then shift our attention to applications of these properties to real-world networks.
\end{abstract}

\maketitle 
%%%%%%%%%%%%%%%%%%%%%%%%%%
\section{Introduction: random graphs and geometry} % (fold)
\label{sec:introduction_random_graphs_and_geometry}

The story of Markov and exponential random graphs relies upon using local behaviours to predict global structure. 
They are described extensively in relation to social network analysis (\cite{pattisonrobins,snijdal,frierand}). 
The key idea is that local interactions (via a \textit{Markov property}) completely determine and are determined by the global behaviour of the probability distribution (a \textit{Gibbs distribution}). We focus on algebraic properties of these models by using a well-developed dictionary between probability distributions and multiaffine polynomials (\cite{bbl,algv1,bh}). Using this dictionary, we focus on two major classes of polynomials: stable polynomials and Lorentzian polynomials.
Both properties imply a form of \textit{negative dependence} of the probability distribution.

 This article focuses on geometric properties of random graph models.
Sections \ref{sec:the_setup_probability_theory_random_fields_and_random_graphs} through \ref{sec:lorentzian_markov_random_graphs} focus on summarizing the major results from \cite{mthesis} and their implications.
Section \ref{sub:algorithms} shines light on applying these results to real-world social networks.

% section introduction_random_graphs_and_geometry (end)

\section{The setup: Probability theory, random fields, and random graphs} % (fold)
\label{sec:the_setup_probability_theory_random_fields_and_random_graphs}

In this section, we lay the groundwork for understanding exponential random graphs. 
For a more thorough introduction, see \cite{mthesis}.
In order to define random graph models, we begin with defining a probability distribution on a finite set.
Let \(\Omega\) be a finite set. 
We define a \textbf{probability distribution} on \(\Omega\) as a function \(P: \Omega \to [0,1]\) such that \(\sum_{\omega \in \Omega} P(\omega) = 1\).
If \(P : \Omega \to (0,1)\), then \(P\) is said to be \textbf{positive}.
As usual, if \(A\seq \Omega\), the \textbf{probability of event \(A\)}
 occurring is \(P(A) \colonequals \sum_{a\in A} P(a)\).
Since we are interested in finite graphs, we define a \textbf{discrete random variable} \(X : \Omega \to \R\)  as a function with finite support.
The probability mass function of a discrete random variable \(X\) is given by \(p(x) = P(X=x)\) for \(x\in \Omega\).

A particularly nice discrete random variable is the \textit{Bernoulli} random variable defined by \(X: \{0,1\} \to \R\) where \(X(i) = i\) and 
\begin{align*}
P(X=x) & = \begin{cases}
 p & x =1\\
 1-p & x = 0	
 \end{cases}
\end{align*}
for some \(p\in [0,1]\).

% In what follows, \(G=(V,E)\) denotes a finite undirected graph with no loops.
\subsection{Markov random fields} % (fold)
\label{sub:markov_random_fields}
Let \(\Omega\) be a finite set. 
Throughout, we will keep an eye towards an underlying finite graph \(G=(V,E)\). In this context, \(\Omega = E\) will be the set of edges of \(G\).
\begin{defn}\label{def:randomfield}
Let \(\omega \in \Omega\). Define \(X_\omega\) as the \textit{Bernoulli random variable} with values in \(\{0,1\}\).
The collection \(X\coloneqq \{X_\omega\}_{\omega \in \Omega}\) is a \textbf{random field} on \(\Omega\) with \textbf{phases} in \(\{0,1\}\). 
	
\end{defn}
We are interested in \textit{Markov} random fields, which encode a \textit{memoryless} dependence through local interactions. 
That is, it involves encoding dependencies between random variables in a random field \(X\). 
Local dependence is made precise via a neighbourhood system.
\begin{defn}\label{def:nbhdsystem}
	A \textbf{neighbourhood system} on \(\Om\), \(\mathcal N \coloneqq \{N_\omega \seq \Om\}_{\omega \in \Om}\), is such that:
	\begin{enumerate}
		\item \textit{Punctured neighbourhoods}: \(\omega \notin N_{\omega}\) for all \(\omega \in \Om\), and
		\item \textit{Symmetric neighbourhoods}: \(\omega_1 \in N_{\omega_2}\) if and only if \(\omega_2 \in N_{\omega_1}\) for all \(\omega_1, \omega_2 \in \Om\).
	\end{enumerate}
\end{defn}
The memoryless local property of Markov random fields can now be made precise.
\begin{defn}\label{def:mrf}
	A \textbf{Markov random field (MRF)} (with respect to \(\mathcal N\)) is a random field \(X\) such that for all \(\omega \in \Om\) and \(x\in \{0,1\}^\Om\)\footnote{The powerset \(\{0,1\}^\Om\) is sometimes called the \textit{configuration space} of \(\Om\) and elements are referred to as \textit{configurations}.},
	\[
		P(X_\omega = x_{\omega} \mid X_{\overline \omega} = x_{\overline \omega}) = P(X_\omega = x_\omega \mid X_{N_{\omega}}= x_{N_\omega})
\]
where \(\overline \omega \coloneqq \Omega \setminus \{\omega\}\) and for \(A\seq \Om\), \(x\in \{0,1\}^\Om\), \(X_A \coloneqq \{X_\omega : \omega \in A\}\).
Further, if \(P(X=x)>0\) for all \(x\in \{0,1\}^\Om\), then \(X\) is a \textbf{positive Markov random field}.
\end{defn}
That is, the conditional probability of the event \(\{X_\omega = x_{\omega}\}\) given all of the information outside of \(\omega\) is the \textit{same as} the conditional probability given only what is in the neighbourhood of \(\omega\).
In other words, it means that the behaviour of \(X_\omega\) can be thought of as being dependent on the behaviour of \(X_{\omega'}\) for \(\omega'\) in the neighbourhood of \(\omega\), and as not dependent on any such element sufficiently far away (i.e., outside of the neighbourhood).

Positive Markov random fields have more dependence woven into their structure, as they are completely determined by their \textit{local specification} map (see \cite{brem}, Theorem 10.1.5).
More remarkably, positive Markov random fields are in a one-to-one correspondence with \textit{finite Gibbs fields} via the \textit{Hammersley-Clifford Theorem}.
We will not delve into the intricacies of Gibbs fields here (see \cite{mthesis} for more detail). 
Instead, we will note that a finite Gibbs field gives an explicit formula for the probability of \(x\in \{0,1\}^\Om\) in the following form\footnote{As standard notation, we use \(\propto\) to denote that the probability is equal to the quantity up to renormalization. In general, we do not consider the normalizing constant in our probability computations because the geometric properties of interest (i.e., stable and Lorentzian) are not influenced by positive scaling.}:
\begin{align*}
	P(X=x) & \propto \exp\left(-\frac 1 T \mathcal E_{\mathcal P}(x)\right)
\end{align*}
where \(T>0\) is a parameter, and \(\mathcal E_{\mathcal P} : x \to \sum_{C\in \mathcal C} \mathcal P_{C} (x_C)\) is the associated \textit{energy function} on the \(\mathcal N\)-cliques \(C\) of \(\{0,1\}^\Om\).
For our purposes, what is important is that this is an \textit{exponential distribution} that relies only on functions defined on some subsets that are determined by a neighbourhood system \(\mathcal N\). 
The general class of random graphs that have this distribution is called an \textbf{exponential random graph}.

We now shift our attention to the case where \(\Omega = E\) is the edge set of a graph \(G=(V,E)\). 
The neighbourhood system used to define a Markov exponential random graph is defined as follows.
\begin{defn}\label{def:markovrandomgraph}
	The \textbf{Markov neighbourhood system} \(\mathcal N\) on \(E\) is defined by \textit{adjacency}. 
	That is,
	\[ N_e \coloneqq \{e' \in E\setminus \{e\} : e,e' \text{ share a vertex}\}.\]
	A Markov \(\mathcal N\)-clique then takes one of two forms: a \(3\)-cycle (i.e., triangle \(\triangle\)) or a \(k\)-star \(S_k\) in \(G\). 
	The corresponding finite Gibbs field is called a \textbf{Markov random graph on \(G\)} with probabilities 
	\[ P(S) \propto  \exp\left(\frac 1 T\left(\beta t(\triangle, G_S) + \sum_{k=1}^{\max \deg(G_S)} \beta_k t(S_k, G_S)\right)\right) \]
	where \(t(H, G) = \frac{\vert \Hom(H,G)\vert}{\vert V(G) \vert^{\vert V(H)\vert}}\) is the \textbf{homomorphism density} of \(H\) in \(G\). 
	We note that \(\vert \Hom(S_k, G)\vert = \sum_{v\in V(G)} \deg(v)^k\) and \(\vert \Hom(\triangle, G)\vert  = 6(\# \text{ triangles in } G)\).
	The graph \(G_S = (V, S)\) is the spanning subgraph of \(G\) defined by the edges in \(S\). 
	This model has parameters \(T>0\) and \(\beta, \beta_k \in \mathbb R\).
\end{defn}
Much of the results and work that follow will be to pinpoint subsets of the parameter space \((T, \beta, \beta_1, \dots, \beta_k)\in \R_{>0} \times \R^{m}\) such that exponential random graphs exhibit some nice properties. 

\begin{ex}\label{ex:cubicdef}
We will now compute the probabilities associated to the complete graph on 3 vertices (the triangle \(K_3\)) under the Markov random graph model.
Suppose we label the edges as follows.
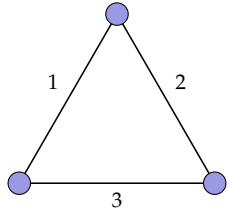
\begin{center}
\begin{tikzpicture}
  \path (90:1.5) edge [semithick, edge label=\footnotesize{1}, swap] (210:1.5)
  		(210:1.5) edge [semithick, edge label=\footnotesize{3}, swap] (330:1.5)
  		(330:1.5) edge [semithick, edge label=\footnotesize{2}, swap] (90:1.5);
  \fill (90:1.5)  node[circle, fill=mblue!50, draw=black, inner sep=3pt] {}
        (210:1.5) node[circle, fill=mblue!50, draw=black, inner sep=3pt] {}
        (330:1.5) node[circle, fill=mblue!50, draw=black, inner sep=3pt] {};
\end{tikzpicture}
\end{center}
The empty set corresponds to the spanning subgraph consisting of 3 isolated vertices. 
As each homomorphsim density is equal to zero, then we have that 
\begin{align*}
P(\varnothing)&\propto \exp\left(\frac 1 T (0)\right) = 1.
\end{align*}
If we consider the spanning subgraph consisting of 1 edge, then \(G_S\) has the degree sequence \((1, 1, 0)\) and hence \(t(S_1, G_S) = \frac{1+1+0}{3^2} =\frac 2 9\).
Therefore, we have the following:
\begin{align*}
P(\{1\}) = P(\{2\}) = P(\{3\}) &\propto \exp\left(\frac{2}{9T} \beta_1 \right).
\end{align*}
Now, if we take any two edges and form the spanning subgraph, we always get a \(2\)-star. Then, \(G_S\) has degree sequence \((2, 1, 1)\) and so we compute the edge and 2-star homomorphism densities as:
\begin{align*}
t(S_1, G_S) & = \frac{2 + 1 +1}{3^2} = \frac{4}{9}\\
t(S_2, G_S) & = \frac{2^2 + 1^2 + 1^2}{3^3} = \frac{6}{27}.
\end{align*}
Thus, the probability of the spanning subgraph consisting of 2 edges is as follows:\footnote{We note that in this special case of \(K_3\), no matter which two edges we pick, it results in a 2-star and so every subset with two edges has the same probability.}.
\begin{align*}
P(\{1,2\})= P(\{1,3\})=P(\{2,3\})&\propto  \exp\left(\frac{4}{9T} \beta_1 + \frac{6}{27T} \beta_2\right).
\end{align*}
Finally, the spanning subgraph with 3 edges is all of \(K_3\) which has degree sequence \((2,2,2)\) and contains 1 triangle.
Hence, the corresponding homomorphism densities are as follows:
\begin{align*}
t(S_1, G_S) & = \frac{2 + 2 +2}{3^2} = \frac{6}{9}\\
t(S_2, G_S) & = \frac{2^2 + 2^2 + 2^2}{3^3} = \frac{12}{27}\\
t(C_3, G_S) & = \frac{6\cdot 1 }{3^3} = \frac{6}{27}.
\end{align*}
Therefore, the probability of the subset with all 3 edges is equal to the following:
\begin{align*}
 P(\{1,2,3\}) &\propto  \exp\left(\frac{6}{9T} \beta_1+ \frac{12}{27T}\beta_2 + \frac{6}{27T}\beta\right).
\end{align*}
\end{ex}

% section the_setup_probability_theory_random_fields_and_random_graphs (end)

\section{At the intersection of probability and geometry} % (fold)
\label{sec:at_the_intersection_of_probability_and_geometry}
In order to describe the aforementioned \textit{nice properties} of exponential random graphs, we make use of a well-developed dictionary between probability distributions and polynomials. 
This was first described in groundbreaking work in 2009 (see \cite{bbl}). 
Let \(P\) be a discrete probability distribution on the power set \(\mathcal P(E)\). 
\begin{defn}
 	The \textbf{generating polynomial of \(P\)}, denoted \(g_P \in \R[\mathbf x^E]\coloneqq \R[x_e : e\in E]\) is
 	\begin{align*}
 		g_P(\mathbf x)&\coloneqq \sum_{S\seq E} P(S) \mathbf x^S
 	\end{align*}
 	where \(\mathbf x ^S \coloneqq \prod_{e\in S} x_e\).
 \end{defn} 
We start with introducing \textit{Erd\H{o}s-R\'enyi} graphs and their corresponding generating polynomials.
\begin{defn}\label{def:bernoulli}
Let \(G=(V,E)\) be a finite graph and $\mathbb{B}:\{0,1\}^{E}\to \R$ be the probability distribution on the set of induced subgraphs of $G$ where an edge $e\in E$ is chosen independently with probability $p_e \in [0,1]$. This is named the \textbf{Bernoulli distribution} (or variant of the \textit{Erd\H{o}s-R\'enyi} model)\footnote{this is sometimes called a \textit{product measure} on $\{0,1\}^{E}$ in certain contexts (such as \cite{bbl}). This is also briefly described in \cite{amini}.} on $G$.
\end{defn}

For ease of notation, let \(\overline S \coloneqq E\setminus S\) denote the complement of the subset \(S\) in \(E\).
\begin{ex}\label{ex:bernoulligen} Let $S\subseteq E$ be an arbitrary subset. As each edge is chosen independently with some fixed probability, the probability of the subset (or spanning subgraph) $S$ is $\mathbb B(S)=\prod_{e\in S}p_e \prod_{e\in \overline S}(1-p_e)$. Then, the generating polynomial is 
\[g_{\mathbb{B}}(\mathbf{x})= \sum_{S\subseteq E} \prod_{e\in S} p_e \prod_{e\in \overline S}(1-p_e) \mathbf{x}^S.\] 
By definition of $\mathbf{x}^S$, then we have the following simplification:
\begin{align*}
g_{\mathbb{B}}(\mathbf{x})&= \sum_{S\subseteq E} \prod_{e\in S} p_ex_e \prod_{e\in \overline S}(1-p_e)
\end{align*}
\begin{equation}\label{eqn:bernoullifactor}
g_{\mathbb B}(\mathbf x) = \prod_{e \in E} (p_ex_e + (1-p_e)).
\end{equation}
We will see later that this has a desirable property (see Proposition \ref{prop:bern-SR}).
\end{ex}
There is also an irreducibility characterization of the generating polynomials of exponential random (hyper)graphs. This is Theorem 4.1 in \cite{mthesis}.

% section at_the_intersection_of_probability_and_geometry (end)

\subsection{Stable polynomials and negative dependence} % (fold)
\label{sub:stable_polynomials_and_negative_dependence}

\begin{defn}\label{def:stable}
A polynomial $g\in \C[\mathbf x^E]$ is said to be \textbf{stable} if either $g\equiv 0$ or it does not have any roots in the open upper half of the complex plane. That is, if \(\mathbf x = (\alpha_e : e\in E)\) and $\im(\alpha_e)>0$ for all $e$, then $g(\mathbf x)\neq 0$. Further, if $g$ has real coefficients, then it is said to be \textbf{real stable}. 
\end{defn}{}

This can be a difficult condition to check directly. 
There is a more tractable method for checking stability, according to Wagner.
\begin{prop}\label{prop:stabcond}\cite{wag}
A multiaffine polynomial $g\in \R[\mathbf x^E]$ is (real) stable if and only if for all $\mathbf{x}\in \R^{E }$ and \(i,j\in E\) such that $i\neq j$,
\[\frac{\partial g}{\partial x_i}(\mathbf{x}) \frac{\partial g}{\partial x_j} (\mathbf{x}) \geq \frac{\partial^2 g}{\partial x_i \partial x_j} (\mathbf{x}) g(\mathbf{x}). \] 
\end{prop}
The corresponding term for the probability distribution is \textit{strongly Rayleigh}.

\begin{defn}\label{def:SR}
A probability measure $P$ is said to be \textbf{strongly Rayleigh} when its corresponding generating polynomial $g_P\in \R[\mathbf x]$ is stable. 
\end{defn}
The strongly Rayleigh property completely characterizes the strongest form of negative dependence for probability distributions. 
This is detailed in \cite{bbl}.

\begin{prop}\label{prop:bern-SR}
The Bernoulli distribution on $G$ is strongly Rayleigh.
\end{prop}
\begin{proof}\ 
Let $S\subseteq E$ be an arbitrary subset. By expression (\ref{eqn:bernoullifactor}), we have that \(g_{\mathbb B}(\mathbf x) = \prod_{e \in E} (p_ex_e + (1-p_e))\).
As $p_e\in \R$, then each polynomial $f_i\coloneqq p_ex_e + (1-p_e) \in \R[x_e]$ is real-rooted (i.e., real stable). The product of real stable polynomials is real stable and hence $g_{\mathbb{B}}$ is real stable. Therefore, $\mathbb{B}$ is strongly Rayleigh.
\end{proof}

% subsection stable_polynomials_and_negative_dependence (end)

\section{Strongly Rayleigh Markov random graphs} % (fold)
\label{sec:strongly_rayleigh_markov_random_graphs}

We can make conclusions about Markov random graphs and their stability. 
The complete proofs can be found in \cite{mthesis}.
The edge-triangle cubic Markov random graph is the model on \(G = K_3\) which only considers edges and triangles as \(\mathcal N\)-cliques.

\begin{prop}\label{prop:cubicmarkovedgetriangle} (\cite{mthesis}, Theorem 5.4)

The edge-triangle cubic Markov random graph is strongly Rayleigh if and only if the triangle parameter \(\beta = 0\).
	
\end{prop}

\begin{prop}\label{prop:cubicmarkovgen} (\cite{mthesis}, Theorem 5.5)

The cubic Markov random graph is strongly Rayleigh if and only if the triangle parameter \(\beta \leq 0\) and the \(2\)-star parameter \(\beta_2\) is such that 
\[\beta_2 = \frac{9 T}{2} \ln\left(3 \exp \left(\frac{2}{9T} \beta\right) - 2 \right) - 2\beta.\]
	
\end{prop}

We can also outline necessary conditions for higher-order Markov random graphs. 
These necessary conditions make use of the negative dependence of strongly Rayleigh distributions (more specifically: the \textit{negative lattice condition}).

\begin{prop}\label{prop:neccondition1} (\cite{mthesis}, Theorem 5.6)

If \(P\) is a strongly Rayleigh Markov random graph on a finite graph with at least one triangle, then the triangle and \(2\)-star parameters \(\beta\) and \(\beta_2\) are such that \(\beta \leq - \beta_2\).
	
\end{prop}

\begin{prop}\label{prop:neccondition2} (\cite{mthesis}, Theorem 5.7)

If \(P\) is a strongly Rayleigh Markov random graph on a finite graph \(G\) where \(G\) has at least one \(3\)-star, then the \(2\)-star and \(3\)-star parameters \(\beta_2\) and \(\beta_3\) are such that \(\beta_3 \leq -\frac 1 5 \vert V(G) \vert \beta_2\).
	
\end{prop}

Further still, we asserted that the edge parameter has no impact on the stability of the model.

\begin{prop} (\cite{mthesis}, Theorem 6.7)

If \(P\) is a Markov random graph on a finite graph, then the edge parameter \(\beta_1\) does not affect the stability of \(g_P\).
	
\end{prop}

Although stability is powerful, results that characterize the stability of generating polynomials are scarce. 
Instead, we now focus our attention on a more tractable condition that still implies some form of negative dependence: the \textit{Lorentzian} property. 

% section strongly_rayleigh_markov_random_graphs (end)

\section{Lorentzian Markov random graphs} % (fold)
\label{sec:lorentzian_markov_random_graphs}
The Lorentzian property was first introduced in 2020 in \cite{bh}.
We will focus on the Lorentzian property, but note that for any interested readers, \cite{algv1} defined the equivalent \textit{completely log-concave} property.

\begin{defn}\label{def:supp}
The \textbf{support} of a polynomial \(f\) is the subset \(\text{supp}(f)\seq \N^n\) defined by:
\[ \text{supp}(f) \coloneqq \{\alpha \in \N^n : c_\alpha \neq 0\}\]
where \(f = \sum_{\alpha \in \N^n} \frac{c_\alpha}{\alpha!} z^\alpha\) and \(\alpha! \coloneqq \prod_{i=1}^n \alpha_i !\).
\end{defn}

\begin{defn}(\cite{bh})\label{def:Mconvex}
A subset \(J\seq \N^n\) is \textbf{\(M\)-convex}\footnote{The term \textit{\(M\)-convex} stems from its connection to the bases of a \textit{matroid}.} if it satisfies any of the following equivalent conditions:

\begin{enumerate}
\item (\textit{Exchange property}). For any \(\alpha, \beta \in J\) and any index \(i\) such that \(\alpha_ i > \beta_i\), there is some index \(j\) such that
\[ \alpha_j < \beta_j \text{ and } \alpha - e_i + e_j \in J.\]
\item (\textit{Symmetric exchange property}). For any \(\alpha, \beta \in J\) and any index \(i\) such that \(\alpha_i > \beta_i\), there is some index \(j\) such that
\[ \alpha_j < \beta_j \text{ and } \alpha - e_i + e_j \in J \text{ and } \beta - e_j + e_i \in J.\]
\item \(J\) is the set of all lattice points of a generalized permutohedron.
\end{enumerate}
\end{defn}

Let \(H_n^d \sub \R[x_1, \dots, x_n]\) be the set of degree \(d\) homogeneous polynomials. Let \(M_n^d\seq H_n^d\) be the set of all degree \(d\) homogeneous polynomials whose supports are \(M\)-convex.
\begin{defn}(\cite{bh})\label{def:lorentz}
Let \(L_n^2 \seq H_n^2\) be the (closed) subset of quadratic forms with nonnegative coefficients that have at most one positive eigenvalue. For \(d>2\), define
\[ L_n^d \coloneqq \{f\in M_n^d : {}\partial_i f \in L_n^{d-1} \text{ for all } i\in [n]\}.\]
Polynomials in \(L_n^d\) are called \textbf{Lorentzian polynomials}.
\end{defn}

\begin{defn}(\cite{bh})\label{def:lorentzprob}
A discrete probability measure \(P\) on \(\mathcal P(E)\) is \textbf{Lorentzian} if the homogenization of its generating polynomial is a Lorentzian polynomial. This homogenization is defined as:
\[ h_P (z,x_1, \dots, x_n) \coloneqq z^n g_P\left(\frac{x_1}{z}, \dots, \frac{x_n}{z}\right).\]
\end{defn}

As promised, this condition is more tractable: To determine when an exponential random graph is Lorentzian, we determine when the partial derivatives that result in quadratic forms have exactly one positive eigenvalue\footnote{Some of the details that make this assertion possible are omitted here for sake of brevity. Interested readers can find more details about this in Chapter 6 of \cite{mthesis}.}.

It is also a weaker condition, as proven in \cite{bh}.

\begin{prop}\label{prop:bhstableimplieslorentzian} (\cite{bh}, Proposition 4.24)

If \(P\) is strongly Rayleigh, then \(P\) is Lorentzian.
	
\end{prop}

\subsection{Lorentzian Markov random graphs} % (fold)
\label{sub:lorentzian_cubic_markov_random_graphs}
We have analogous results about Lorentzian Markov random graphs as we did for the strongly Rayleigh conditions.

\begin{prop}\label{prop:cubicedgetrianglelorentz} (\cite{mthesis}, Theorem 6.1)

The cubic edge-traingle Markov random graph is Lorentzian if and only if the triangle parameter \(\beta \leq 0\).
	
\end{prop}

\begin{prop}\label{prop:cubiclorentzgen} (\cite{mthesis}, Theorem 6.2)

The cubic Markov random graph is Lorentzian if and only if the \(2\)-star and triangle parameters \(\beta_2, \beta\) are such that \(\beta_2, \beta \leq 0\).
	
\end{prop}
% subsection lorentzian_cubic_markov_random_graphs (end)

% section lorentzian_markov_random_graphs (end)

\section{Applying the theory to real networks} % (fold)
\label{sub:algorithms}

To make conclusions about higher-order Markov random graphs, we designed and implemented an algorithm to identify the Lorentzian property of the Markov random graph on \(K_n\). 
Details of the algorithm and its results are in Chapter 7 of \cite{mthesis}.

Now, we outline how we can use exponential random graphs and the strongly Rayleigh property to make conclusions about real-world networks. 
Historically, Markov random graphs (and more generally: \textit{exponential} random graphs) are used to model social networks (see \cite{pattisonrobins,rpkl,jackson}). 
To make conclusions about the strongly Rayleigh property in social networks, we look at some commonly used social network datasets.

First, we consider a social network of business connections among 16 families from the 15th century in Florence, Italy (as discussed in \cite{rpkl}). 
We call this social network the Medici business social network. 
This is one of the first-known examples of a social network.
\begin{figure}[H]\label{fig:medici}

\centering
         \includesvg[width=3in]{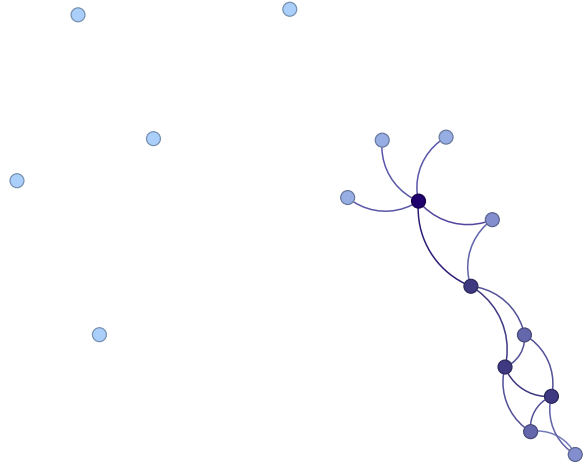}

\vspace*{5mm}
\caption{Medici business social network. Vertex-colouring is based on the degree distribution.}
\end{figure}

After conducting Markov chain Monte Carlo (MCMC) maximum likelihood estimation for the Markov random graph parameters (using the \href{https://cran.r-project.org/web/packages/ergm/index.html}{\texttt{ergm}} package in R), we obtain the following estimates for the parameters:

\begin{align*}
	\beta_1 & = -4.2858\\
	\beta_2 &= 1.0611\\
	\beta_3 & = -0.6339 \\
	\beta & = 1.3126
\end{align*}

In particular, we notice that \(\beta = 1.3126 > -(1.0611) = - \beta_2\). By Proposition \ref{prop:neccondition1}, then the Markov random graph model on this social network is \textit{not} strongly Rayleigh. 

\begin{prop}\label{prop:mediciSR}
	The Markov random graph model on the Medici business social network is not strongly Rayleigh.
\end{prop}

\newpage
The second real-world network we analyze is the Sampson dataset. This is a network that displays community structure in a New England monastery with \(n=18\) vertices (this dataset is  discussed in \cite{frankstrauss} and \cite{snijdersmc}). 

\begin{figure}[H]\label{fig:sampson}

\centering
         \includesvg[width=3in]{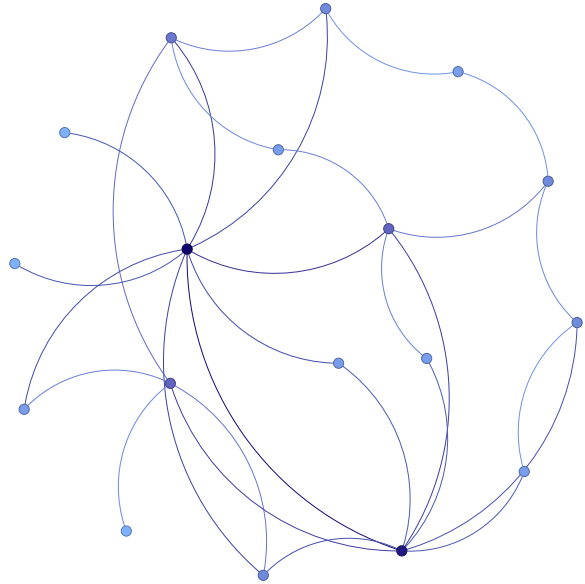}

\vspace*{5mm}{}
\caption{Sampson monastery social network. Vertex-colouring is based on the degree distribution.}
\end{figure}
% \vspace*{-15mm}
The MCMC maximum likelihood estimation yields the following parameters for the Markov random graph model.
\begin{align*}
	\beta_1 & = -0.78 \\
	\beta_2 & = -0.05 \\
	\beta & = 0.35
\end{align*}
Similar to the Medici network, we note that \(\beta = 0.35 > -(-0.05) = - \beta_2\). By Proposition \ref{prop:neccondition1}, the Markov random graph model is not strongly Rayleigh.
\begin{prop}\label{prop:sampsonSR}
	The Markov random graph model on the Sampson monastery social network is not strongly Rayleigh.
\end{prop}

\newpage
Next, we consider a network of colleagues (work ties) at a law firm where there are \(n= 36\) lawyers (see \cite{lazega} and \cite{pattisonrobins}). 

\begin{figure}[H]\label{fig:lazega}

\centering
         \includesvg[width=3in]{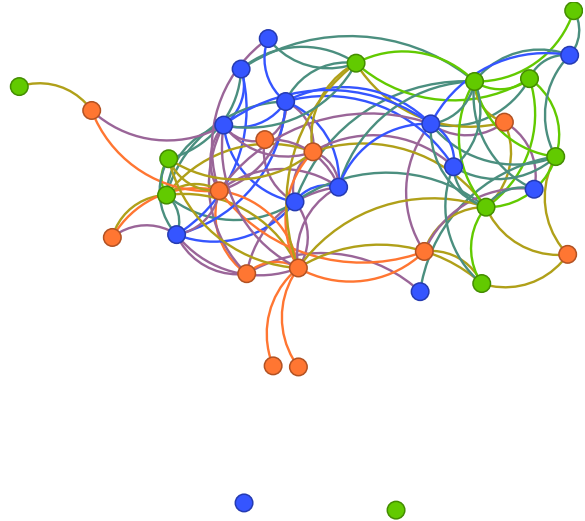}

\vspace*{2mm}{}
\caption{Lazega law firm social network. Vertex-colouring is based on clustering the lawyers' law schools.}
\end{figure}
MCMC maximum likelihood estimation results in the following parameter estimates.
\begin{align*}
	\beta_1 & = -2.79 \\
	\beta_2 & = -0.02 \\
	\beta & = 0.48
\end{align*}
Similar to the previous networks, we have that \(\beta = 0.48 > -(-0.02) = -\beta_2\).
Therefore, by Proposition \ref{prop:neccondition1}, the Markov random graph model is not strongly Rayleigh.

\begin{prop}\label{prop:lawfirm}
	The Markov random graph model on the Lazega law firm network is not strongly Rayleigh.
\end{prop}

\newpage
Finally, we analyze a network of friendships in a bank wiring room with \(n=14\) employees (see \cite{roeth} and \cite{pattisonrobins}). 

\begin{figure}[H]\label{fig:bankwiring}

\centering
         \includesvg[width=3in]{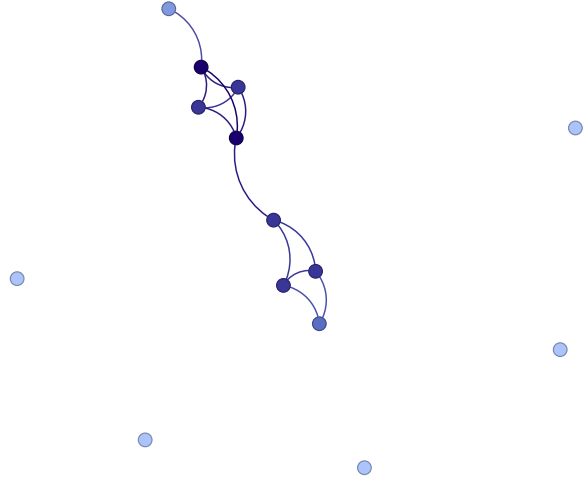}

\vspace*{2mm}{}
\caption{Bank wiring friendship network. Vertex-colouring is based on the degree distribution.}
\end{figure}

MCMC maximum likelihood estimation results in the following parameter estimates.
\begin{align*}
	\beta_1 & = -2.66 \\
	\beta_2 & = -0.29 \\
	\beta & = 3.19
\end{align*}
Similar to the previous networks, we have that \(\beta = 3.19 > -(-0.29) = -\beta_2\).
Therefore, by Proposition \ref{prop:neccondition1}, the Markov random graph model is not strongly Rayleigh.

\begin{prop}\label{prop:bankwiring}
	The Markov random graph model on the bank wiring friendship network is not strongly Rayleigh.
\end{prop}

For all of the above networks, the lack of the strongly Rayleigh property implies that they do not exhibit strong negative dependence.
This is not surprising if we consider the implication of negative dependence when modelling a social network.
Suppose that Alice and Bob are both friends with Mary. 
If the model exhibited negative dependence, this would mean that the probability of Alice and Bob becoming friends is \textit{smaller than} the product of the individual edge probabilities of the pairs Alice and Mary and Alice and Bob and the two-star probability between Alice, Bob, and Mary.
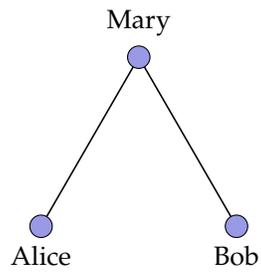
\begin{figure}[H]
\centering
\begin{tikzpicture}
  \path (90:1.5) edge [semithick] (210:1.5)
  		%(210:1.5) edge [semithick] (330:1.5)
  		(330:1.5) edge [semithick] (90:1.5);
  \fill (90:1.5)  node[circle, fill=mblue!50, draw=black, inner sep=3pt, label={Mary}] {}
        (210:1.5) node[circle, fill=mblue!50, draw=black, label={[xshift=0cm, yshift=-0.8cm]Alice}, inner sep=3pt] {}
        (330:1.5) node[circle, fill=mblue!50, draw=black, label={[xshift=0cm, yshift=-0.8cm]Bob}, inner sep=3pt] {};
\end{tikzpicture}
\end{figure}
This goes against much of what we understand about social networks and the transitivity effect.
In most cases, friends of friends are \textit{more} likely to become friends themselves at some point.  

Future work in this area would primarily be to consider other types networks and the presence or absence of the strongly Rayleigh property in those cases.
% section algorithms (end)

\np
\bibliographystyle{alpha}
\bibliography{references}

\begin{thebibliography}{RPKL07}

\bibitem[AGV21]{algv1}
N.~Anari, S.~O. Gharan, and C.~Vinzant.
\newblock Log-concave polynomials i: Entropy and a deterministic approximation
  algorithm for counting bases of matroids.
\newblock {\em Duke Mathematical Journal}, October 2021.

\bibitem[Ami19]{amini}
N.~Amini.
\newblock Stable multivariate generalizations of matching polynomials.
\newblock {\em \tt \href{https://arxiv.org/abs/1905.02264}{arXiv:1905.02264
  [math.CO]}}, May 2019.

\bibitem[BBL09]{bbl}
J.~Borcea, P.~Br{\"a}nd{\'e}n, and T.M. Liggett.
\newblock Negative dependence and the geometry of polynomials.
\newblock {\em Journal of the American Mathematical Society}, 22(2):521--567,
  April 2009.

\bibitem[BH20]{bh}
P.~Br{\"a}nd{\'e}n and J.~Huh.
\newblock Lorentzian polynomials.
\newblock {\em Annals of Mathematics}, 192(3):821--891, November 2020.

\bibitem[Br{\'e}20]{brem}
P.~Br{\'e}maud.
\newblock {\em Markov chains: Gibbs fields, Monte Carlo simulation, and
  queues}.
\newblock Springer Nature Switzerland AG, second edition, May 2020.

\bibitem[FK15]{frierand}
A.~Frieze and M.~Karo{\'n}ski.
\newblock {\em Introduction to random graphs}.
\newblock Cambridge University Press, November 2015.

\bibitem[FS86]{frankstrauss}
O.~Frank and D.~Strauss.
\newblock Markov graphs.
\newblock {\em Journal of the American Statistical Association},
  81(395):832--842, September 1986.

\bibitem[Jac10]{jackson}
M.~O. Jackson.
\newblock {\em Social and economic networks}, chapter~4, pages 77--121.
\newblock Princeton University Press, 2010.

\bibitem[LVD97]{lazega}
E.~Lazega and M.~Van~Dujin.
\newblock Position in formal structure, personal characteristics and choices of
  advisors in a law firm: A logistic regression model for dyadic network data.
\newblock {\em Social Networks: An International Journal of Structural
  Analysis}, 19:375--397, 1997.

\bibitem[PR02]{pattisonrobins}
P.~Pattison and G.~Robins.
\newblock Neighborhood-based models for social networks.
\newblock {\em Sociological Methodology}, 32:301--337, 2002.

\bibitem[RD39]{roeth}
F.~J. Roethlisberger and W.~J. Dickson.
\newblock {\em Management and the worker}.
\newblock Harvard University Press, Cambridge, MA, 1939.

\bibitem[RPKL07]{rpkl}
G.~Robins, P.~Pattison, Y.~Kalish, and D.~Lusher.
\newblock An introduction to exponential random graph (p*) models for social
  networks.
\newblock {\em Social Networks: An International Journal of Structural
  Analysis}, 29:173--191, May 2007.

\bibitem[Sni02]{snijdersmc}
T.~A.~B. Snijders.
\newblock Markov chain monte carlo estimation of exponential random graph
  models.
\newblock April 2002.

\bibitem[SPH06]{snijdal}
T.~A.~B. Snijders, P.~E. Pattison, and M.~S. Handcock.
\newblock New specifications for exponential random graph models.
\newblock {\em Sociological Methodology}, 36:99--153, 2006.

\bibitem[Tar23]{mthesis}
M.~Tarkeshian.
\newblock {\em Generating polynomials of exponential random graphs}.
\newblock PhD thesis, The University of Western Ontario, August 2023.

\bibitem[Wag09]{wag}
D.~Wagner.
\newblock Multivariate stable polynomials: theory and applications.
\newblock {\em Bulletin of the American Mathematical Society}, 48(1), November
  2009.

\end{thebibliography}

\end{document}